\documentclass[11pt]{article}
\usepackage{amsmath,amsthm,amssymb}
\usepackage[dvips]{graphicx}

\newcommand{\N}{\mathbb N}

\newcommand{\R}{\mathbb R}
\newcommand{\C}{\mathbb C}

\renewcommand{\P}{\mathbb P}

\newcommand{\sgn}{\operatorname{sgn}}

\newcommand{\cc}{^{\operatorname{c}}}
\newcommand{\tS}{\tilde{S}}
\newcommand{\tR}{\tilde{R}}

\begin{document}


\title{A Lower Bound for the Norm of the Minimal Residual Polynomial\footnote{published in: Constructive Approximation {\bf 33} (2011), 425--432.}}
\author{Klaus Schiefermayr\footnote{University of Applied Sciences Upper Austria, School of Engineering and Environmental Sciences, Stelzhamerstrasse\,23, 4600 Wels, Austria, \textsc{klaus.schiefermayr@fh-wels.at}}}
\date{}
\maketitle

\theoremstyle{plain}
\newtheorem{theorem}{Theorem}
\newtheorem{corollary}{Corollary}
\newtheorem{lemma}{Lemma}
\newtheorem{definition}{Definition}
\theoremstyle{definition}
\newtheorem*{remark}{Remark}

\begin{abstract}
Let $S$ be a compact infinite set in the complex plane with $0\notin{S}$, and let $R_n$ be the minimal residual polynomial on $S$, i.e., the minimal polynomial of degree at most $n$ on $S$ with respect to the supremum norm provided that $R_n(0)=1$. For the norm $L_n(S)$ of the minimal residual polynomial, the limit $\kappa(S):=\lim_{n\to\infty}\sqrt[n]{L_n(S)}$ exists. In addition to the well-known and widely referenced inequality $L_n(S)\geq\kappa(S)^n$, we derive the sharper inequality $L_n(S)\geq2\kappa(S)^n/(1+\kappa(S)^{2n})$ in the case that $S$ is the union of a finite number of real intervals. As a consequence, we obtain a slight refinement of the Bernstein--Walsh Lemma.
\end{abstract}

\noindent\emph{Mathematics Subject Classification (2000):} 41A17, 41A29, 65F10

\noindent\emph{Keywords:} Bernstein--Walsh lemma, Estimated asymptotic convergence factor, Inequality, Inverse polynomial image, Minimal residual polynomial, Minimum deviation

\section{Introduction}


Let $S$ be a compact infinite set in the complex plane with $0\notin{S}$, and let $\P_n$ denote the set of all polynomials of degree at most $n$ with complex coefficients. For a polynomial $P_n\in\P_n$, let the supremum norm $\|\cdot\|_S$ associated with $S$ be defined by
\[
\|P_n\|_S:=\max_{z\in{S}}|P_n(z)|.
\]
Consider the following approximation problem: Find that polynomial $R_n\in\P_n$ with $R_n(0)=1$, for which the supremum norm on $S$ is minimal, i.e.,
\begin{equation}\label{Ln}
L_n(S):=\|R_n\|_S=\min\bigl\{\|P_n\|_S:P_n\in\P_n,P_n(0)=1\bigr\}.
\end{equation}
The optimal polynomial $R_n\in\P_n$ is unique and called the \emph{minimal residual polynomial} for the degree $n$ on $S$, and the quantity $L_n(S)$ is called the minimum deviation of $R_n$ on $S$. Note that we say \emph{for} the degree $n$ but not \emph{of} degree $n$ since it turns out that the minimal residual polynomial for the degree $n$ is a polynomial of degree $n$ or $n-1$ if $S$ is a real set. It is known, see \cite{Kuijlaars} or \cite{DTT}, that the limit
\begin{equation}\label{kappa}
\kappa(S):=\lim_{n\to\infty}\sqrt[n]{L_n(S)}
\end{equation}
exists, where $\kappa(S)$ is usually called the \emph{estimated asymptotic convergence factor}. This factor $\kappa(S)$ has a very nice representation in  terms of the corresponding Green's function for the complement of the set $S$, see the beginning of Section\,3 and in particular formula \eqref{kappa-g}.


The approximation problem \eqref{Ln} and the convergence factor \eqref{kappa} arise for instance in the context of solving large linear systems $Ax=b$ by Krylov subspace methods, where the spectrum of $A$ is approximated by the set $S$. There is an enormous literature on this subject, hence we would like to mention only three references, the review of Discroll, Toh and Trefethen\,\cite{DTT}; the book of Fischer\,\cite{Fischer-Book}; and the review of Kuijlaars\,\cite{Kuijlaars}.


For $S\subset\C$, the lower bound
\[
L_n(S)\geq\kappa(S)^n
\]
is well known and widely referenced, see \cite{Kuijlaars} or \cite{DTT}, and can be proven with the help of the Bernstein-Walsh lemma, see \cite[Sec.\,5.5]{Ransford}. In the present paper, for the union of a finite number of real intervals, we obtain the sharper lower bound
\begin{equation}\label{LB0}
L_n(S)\geq\frac{2\,\kappa(S)^n}{1+\kappa(S)^{2n}},
\end{equation}
and, in addition, we give sets $S$ for which equality is attained in \eqref{LB0}. Moreover, from the proof of \eqref{LB0}, we obtain a refinement for the Bernstein-Walsh lemma~\cite[Thm.\,5.5.7]{Ransford} for the case of several real intervals and real arguments.


The paper is organized as follows. In Section\,2, we prove some essential properties of the minimal residual polynomial on a real set $S$, whereas the main result \eqref{LB0} and the refinement of the Bernstein-Walsh lemma are stated and proved in section\,3.

\section{Properties of Minimal Residual Polynomials on Real Sets}


Let us start with a necessary and sufficient alternation criterion for a polynomial to be the minimal residual polynomial on a real compact set. This criterion can easily be transferred from the results of Achieser\,\cite[Sec.\,1]{Achieser-1932} for the two interval case into the more general case of a compact real set, see \cite[Char.\,4.2]{Fischer-1992} and in particular \cite[Cor.\,3.1.4]{Fischer-Book}.


\begin{theorem}[Achieser\,\cite{Achieser-1932}]\label{Thm-Char}
Let $n\in\N$, and let $S\subset\R$ be compact with at least $n+1$ points and $0\notin{S}$. The polynomial $P_n\in\P_n$ is the minimal residual polynomial on $S$ if and only if there exist $n+1$ points $x_0,x_1,\ldots,x_n\in{S}$, $x_0<x_1<\ldots<x_n$, such that $|P_n(x_j)|=\|P_n\|_S$, $j=0,1,\ldots,n$, and
\begin{equation}\label{eq-delta}
P_n(x_{j})=(-1)^{\delta_j+1}P_n(x_{j+1}),\qquad j=0,1,\ldots,n-1,
\end{equation}
where $\delta_j=1$ if $x_j<0<x_{j+1}$ and $\delta_j=0$ otherwise.
\end{theorem}


In the following, we will need the notion of inverse polynomial images. As usual, for a polynomial $P_n\in\P_n\setminus\P_{n-1}$, let us denote by
\[
P_n^{-1}([-1,1]):=\bigl\{z\in\C:P_n(z)\in[-1,1]\bigr\}
\]
the inverse image of $[-1,1]$ with respect to the polynomial mapping $P_n$. Inverse polynomial images are often used in the literature, see, e.g., \cite{Aptekarev}, \cite{Peh1996}, \cite{Peh2003}, \cite{Totik-2001}, \cite{Totik-2002} and \cite{Totik-2009}. In general, $P_n^{-1}([-1,1])$ consists of $n$ Jordan arcs in the complex plane, on which $P_n$ is strictly monotone increasing from $-1$ to $+1$, see \cite{Peh2003}. Of special interest in the context of the present paper is the case when $P_n^{-1}([-1,1])$ is a subset of the real line. The next lemma, due to Peherstorfer\,\cite[Cor.\,2.3]{Peh2003}, gives a necessary and sufficient condition for this case.


\begin{lemma}[Peherstorfer\,\cite{Peh2003}]\label{Lemma-RealInverseImage}
Let $P_n\in\P_n\setminus\P_{n-1}$, $n\in\N$, and let $A:=P_n^{-1}([-1,1])\subset\C$. Then $A\subset\R$ if and only if all coefficients of $P_n$ are real, $P_n$ has $n$ simple real zeros and $\min\{|P_n(y)|:P_n'(y)=0\}\geq1$. In case of $A\subset\R$, $A$ consists of $\ell$ finite intervals with $1\leq\ell\leq{n}$, where $\ell-1$ is the number of the zeros $y$ of $P_n'$, for which $|P_n(y)|>1$.
\end{lemma}


Theorem\,\ref{Thm-Char} together with Lemma\,\ref{Lemma-RealInverseImage} gives the following result:


\begin{corollary}\label{Cor_InvPolIm}
Let $P_n\in\P_n\setminus\P_{n-1}$, $n\in\N$, such that $A:=P_n^{-1}([-1,1])$ is a subset of the real line and $0\notin{A}$. Then $R_n(x):=P_n(x)/P_n(0)$ is the minimal residual polynomial for the degree $n$ on $A$ with minimum deviation $L_n(A)=1/|P_n(0)|$. In addition, if $0$ is in the convex hull of $A$, then $R_n(x)$ is the minimal residual polynomial also for the degree $n+1$ on $A$, i.e., $L_{n+1}(A)=L_n(A)$.
\end{corollary}
\begin{proof}
Let $P_n\in\P_n\setminus\P_{n-1}$, $A:=P_n^{-1}([-1,1])$, and suppose that $A\subset\R$ with $0\notin{A}$. By Lemma\,\ref{Lemma-RealInverseImage}, $P_n$ has $n$ simple real zeros, hence $P_n'$ has $n-1$ simple real zeros. Let $\ell-1$ be the number of zeros $\xi_i$ of $P_n'$ for which $|P_n(\xi_i)|>1$. Then, again by Lemma\,\ref{Lemma-RealInverseImage}, the remaining $n-\ell$ zeros $\zeta_i$ of $P_n'$ satisfy $|P_n(\zeta_i)|=1$, and the set $A$ is the union of $\ell$ finite intervals, say $A=\bigcup_{j=1}^{\ell}[a_{2j-1},a_{2j}]$, $a_1<a_2<\ldots<a_{2\ell}$. Now, if the point $0$ is outside the convex hull of $A$, then the set of $n+1$ points $\{x_0,x_1,\ldots,x_n\}$ which satisfies \eqref{eq-delta} is given by $\{\zeta_1,\ldots,\zeta_{n-\ell}\}\cup\{a_1\}\cup\{a_2,a_4,\ldots,a_{2\ell}\}$. If the point $0$ lies in the convex hull of $A$, say $0\in(a_{2j^*},a_{2j^*+1})$, then the set of $n+2$ points $\{\zeta_1,\ldots,\zeta_{n-\ell}\}\cup\{a_1\}\cup\{a_2,a_4,\ldots,a_{2\ell}\}\cup\{a_{2j^*+1}\}$ satisfies condition \eqref{eq-delta}.
\end{proof}


Next, we prove that the inverse image of any (suitable normed) minimal residual polynomial on a real set is again a real set.


\begin{lemma}\label{Lemma-InvImageR}
Let $n\in\N$, let $S\subset\R$ be compact with at least $n+1$ points and $0\notin{S}$, let $R_n\in\P_n$ be the minimal residual polynomial on $S$ for the degree $n$ with minimum deviation $L_n$, and define $P_n:=R_n/L_n$. Then $A:=P_n^{-1}([-1,1])$ is the union of $\ell$ finite disjoint real intervals, where $1\leq\ell\leq{n}$ and $S\subseteq{A}\subset\R$.
\end{lemma}
\begin{proof}
Since $S$ is a real set, the minimal residual polynomial $R_n$ has real coefficients, i.e., for $x\in\R$ one has $R_n(x)\in\R$. Let $x_0,x_1,\ldots,x_n\in{S}$, $x_0<x_1<\ldots<x_n$, be the $n+1$ points which satisfy $|R_n(x_j)|=\|R_n\|_S$, $j=0,1,\ldots,n$, and relation\,\eqref{eq-delta} of Theorem\,\ref{Thm-Char}. We have to distinguish two cases:
\begin{enumerate}
\item[1.] $0\notin[x_0,x_n]$: In this case, the parameter $\delta_j$ is zero for all $j=0,1,\ldots,n-1$, thus $R_n$ has $n$ simple zeros $\xi_1,\xi_2,\ldots,\xi_n$ with
\[
x_0<\xi_1<x_1<\xi_2<x_2<\ldots<x_{n-1}<\xi_n<x_n.
\]
Hence, for each $j\in\{1,2,\ldots,n-1\}$, there is a point $y_j\in(\xi_j,\xi_{j+1})$ with $R_n'(y_j)=0$ and $|R_n(y_j)|\geq{L}_n$, which together with Lemma\,\ref{Lemma-RealInverseImage} gives the assertion.
\item[2.] There exists $j^*\in\{0,1,\ldots,n-1\}$ such that $x_{j^*}<0<x_{j^*+1}$. Then, by relation\,\eqref{eq-delta}, $R_n$ has $n-1$ zeros $\xi_1,\xi_2,\ldots,\xi_{n-1}$, for which
\[
x_0<\xi_1<x_1<\ldots<\xi_{j^*}<x_{j^*}<x_{j^*+1}<\xi_{j^*+1}<\ldots<x_{n-1}<\xi_{n-1}<x_n.
\]
Thus, there are $n-2$ points $y_j$ with $\xi_j<y_j<\xi_{j+1}$, $R_n'(y_j)=0$ and $|R_n(y_j)|\geq{L}_n$, $j=1,2,\ldots,n-2$. Now, there are two possibilities:
\begin{enumerate}
\item[2.1] The minimal residual polynomial $R_n$ is a polynomial of degree $n-1$.\\
Then all $n-1$ zeros $y_j$ of the derivative $R_n'$ are found, and, by Lemma\,\ref{Lemma-RealInverseImage}, $P_n^{-1}([-1,1])$ is the union of a finite number of (at most $n-1$) real intervals.
\item[2.2] The minimal residual polynomial $R_n$ is a polynomial of degree $n$.\\
We claim that the above constructed $n-1$ zeros $\xi_1,\xi_2,\ldots,\xi_{n-1}$ of $R_n$ are all simple and that there are no other zeros of $R_n$ in the interval $[x_0,x_n]$. Assume that one zero, say $\xi_k$, is a double zero of $R_n$; then it is a simple zero of $R_n'$, and therefore, in addition to this zero and the $n-2$ zeros $y_1,\ldots,y_{n-2}$ of $R_n'$, there is another zero of $R_n'$ in the interval $(\xi_k,x_k)$ (if $k\leq{j}^*$) or $(\xi_k,x_{k+1})$ (if $k>j^*$), respectively, which is a contradiction. One gets an analogous contradiction when assuming that, besides $\xi_1,\xi_2,\ldots,\xi_{n-1}$, there is another zero of $R_n$ in $[x_0,x_n]$.\\
Hence, there is a simple zero $\xi^*$ of $R_n$ in $\R\setminus[x_0,x_n]$. If $\xi^*<x_0$ or $x_n<\xi^*$, then there is a zero $y^*$ of $R_n'$ in $(\xi^*,\xi_1)$ or $(\xi_{n-1},\xi^*)$, respectively, for which $|R_n(y^*)|\geq{L}_n$.\\
Altogether, we have $n-1$ simple zeros $y$ of $R_n'$ with $|R_n(y)|\geq{L}_n$, which, by Lemma\,\ref{Lemma-RealInverseImage}, gives the assertion.
\end{enumerate}
\end{enumerate}
\end{proof}


An immediate consequence of the proof of Lemma\,\ref{Lemma-InvImageR} is the following corollary:


\begin{corollary}\label{Cor-nn}
Let $n\in\N$, and let $S\subset\R$ be compact with at least $n+1$ points and $0\notin{S}$. Then the minimal residual polynomial $R_n\in\P_n$ for the degree $n$ on $S$ is a polynomial of degree $n$ or of degree $n-1$.
\end{corollary}

\section{An Inequality for the Minimum Deviation}


Let us introduce the notion of Green's function \cite{Ransford}. Let $S$ be the union of $\ell$ intervals, i.e.,
\begin{equation}\label{S}
S:=[a_1,a_2]\cup[a_3,a_4]\cup\ldots\cup[a_{2\ell-1},a_{2\ell}],
\end{equation}
with $a_1<a_2<\ldots<a_{2\ell}$. Then there exists a (uniquely determined) Green's function for $S\cc:=\overline{\C}\setminus{S}$ (where $\overline{\C}:=\C\cup\infty$) with pole at infinity, denoted by $g(z;S\cc,\infty)$ or, shorter, $g(z;S\cc)$. The Green's function is defined by the following three properties:
\begin{itemize}
\item $g(z;S\cc)$ is harmonic in $S\cc$;
\item $g(z;S\cc)-\log|z|$ is harmonic in a neighbourhood of infinity;
\item $g(z;S\cc)\to0$ as $z\to{S}$, $z\in{S}\cc$.
\end{itemize}
The Green's function $g$ has the following montonicity property: If
\[
\tilde{S}:=[\tilde{a}_1,\tilde{a}_2]\cup[\tilde{a}_3,\tilde{a}_4]
\cup\ldots\cup[\tilde{a}_{2\tilde{\ell}-1},\tilde{a}_{2\tilde{\ell}}],
\]
$\tilde{a}_1<\tilde{a}_2<\ldots<\tilde{a}_{2\tilde{\ell}}$, and $S\subset\tilde{S}$, i.e., $S\cc\supset\tilde{S}\cc$, then
\begin{equation}\label{Ineq-g}
g(z;S\cc)>{g}(z;\tS\cc)\qquad\text{for~all}~z\in\tilde{S}\cc.
\end{equation}
With the Green's function $g(z;S\cc)$, the estimated asymptotic convergence factor $\kappa(S)$ defined in \eqref{kappa} can be characterized by
\begin{equation}\label{kappa-g}
\kappa(S)=\exp(-g(0;S\cc)).
\end{equation}
This connection was first observed by Eiermann, Li and Varga\,\cite{ELV}, see also \cite{Kuijlaars} and \cite{DTT}.


Next, let us recall a result of Peherstorfer\,\cite{Peh1996} concerning the representation of the Green's function for the complement of inverse polynomial images, which we will need for the proof of our main result.


\begin{lemma}[Peherstorfer\,\cite{Peh1996}]\label{Lemma-GreenFunction}
Let $P_n\in\P_n\setminus\P_{n-1}$ be a polynomial of degree $n$, and let $A:=P_n^{-1}([-1,1])$. Then the Green's function for $A\cc:=\overline{\C}\setminus{A}$ is given by
\[
g(z,A\cc)=\tfrac{1}{n}\log\bigl|P_n(z)+\sqrt{P_n^2(z)-1}\bigr|,
\]
where for $\sqrt{\quad}$ that branch is chosen for which $\sgn\sqrt{x^2-1}=\sgn(x-1)$, $x\in\R\setminus[-1,1]$.
\end{lemma}


Now we are ready to state and prove our main result.


\begin{theorem}\label{Thm-LB}
Let $n\in\N$, and let $S$ be the union of a finite number of real intervals as in \eqref{S} with $0\notin{S}$. Let $L_n(S)$ and $\kappa(S)$ as in \eqref{Ln} and \eqref{kappa}, respectively. Then the inequality
\begin{equation}\label{LB}
L_n(S)\geq\frac{2\kappa(S)^n}{1+\kappa(S)^{2n}}
\end{equation}
holds. Equality is attained in \eqref{LB} if and only if there exists a polynomial $P_n\in\P_n\setminus\P_{n-1}$ of degree $n$ such that $S=P_n^{-1}([-1,1])$. If, in addition to $S=P_n^{-1}([-1,1])$, the point zero lies in the convex hull of $S$, then
\[
L_n(S)=L_{n+1}(S).
\]
\end{theorem}


\begin{proof}
Let $R_n$ be the minimal residual polynomial for the degree $n$ on $S$ with minimum deviation $L_n\equiv{L}_n(S)<1$. Define $\tR_n(x):=R_n(x)/L_n$, thus $\tR_n(0)=1/L_n$. Define $\tS:=\tR_n^{-1}([-1,1])$; then, by Lemma\,\ref{Lemma-InvImageR}, $S\subseteq\tS\subset\R$ and $0\notin\tS$. By Corollary\,\ref{Cor-nn}, $\tR_n$ is a polynomial of degree $m$, where $m=n$ or $m=n-1$. By Lemma\,\ref{Lemma-GreenFunction}, the Green's function for $\tS\cc$ is
\begin{equation}\label{GreenFunction}
g(z;\tS\cc)=\tfrac{1}{m}\log\bigl|\tR_n(z)+\sqrt{\tR_n^2(z)-1}\bigr|,\qquad{z}\in\tS\cc.
\end{equation}
Since $S\cc\supseteq\tS\cc$ and $0\in\tS\cc$, by \eqref{Ineq-g},
\begin{equation}\label{Ineq-g0}
g(0;S\cc)\geq{g}(0;\tS\cc).
\end{equation}
Hence,
\begin{align*}
\eqref{GreenFunction}~\text{and}~\eqref{Ineq-g0}&\Rightarrow
g(0;S\cc)\geq\tfrac{1}{m}\log\bigl(\tR_n(0)+\sqrt{\tR_n^2(0)-1}\bigr)\\
&\Rightarrow\exp(m\cdot{g}(0;S\cc))\geq\frac{1}{L_n}+\sqrt{\frac{1}{L_n^2}-1}\\
&\Rightarrow\frac{1}{L_n}\leq\frac{1}{2}\bigl(\exp(m\cdot{g}(0;S\cc))+\exp(-m\cdot{g}(0;S\cc))\bigr)\\
&\Rightarrow{L_n}\geq\frac{2}{\kappa(S)^m+\kappa(S)^{-m}}
=\frac{2\kappa(S)^m}{1+\kappa(S)^{2m}}\geq_{(*)}\frac{2\kappa(S)^n}{1+\kappa(S)^{2n}},
\end{align*}
and inequality \eqref{LB} is proven.

Now suppose that there exists a polynomial $P_n$ of degree $n$ such that $S=P_n^{-1}([-1,1])$. Then, by Corollary\,\ref{Cor_InvPolIm}, $R_n(x)=P_n(x)/P_n(0)$ is the minimal residual polynomial with minimum deviation $L_n\equiv{L}_n(S)=1/|P_n(0)|$, where in addition $L_{n+1}(S)=L_n(S)$ holds if $0$ lies in the convex hull of $S$. Thus $\tR_n(x)=R_n(x)/L_n=\pm{P}_n(x)$ and $\tS=S$, hence equality is attained in \eqref{Ineq-g0}. Since $m=n$, equality is attained in $(*)$. Thus, altogether, equality is attained in \eqref{LB}.

On the other hand, suppose that there is no polynomial $P_n$ of degree $n$ such that $S=P_n^{-1}([-1,1])$. Then either (i) $\tS:=\tR_n^{-1}([-1,1])=S$ and $\tR_n$ is a polynomial of degree $n-1$; or (ii) $\tS\supset{S}$. In case (i), the ``$\geq$'' in $(*)$ is in fact a ``$>$'' and we get $L_n(S)>2\kappa(S)^n/(1+\kappa(S)^{2n})$. In case (ii), by \eqref{Ineq-g} and the above part of the proof, we get again $L_n(S)>2\kappa(S)^n/(1+\kappa(S)^{2n})$.
\end{proof}


\begin{remark}
\begin{enumerate}
\item Let us mention that the factor $2$, which occurs in inequality \eqref{LB}, also appears when estimating the norm of the $n$-th Chebyshev polynomial on a compact real set \cite{Sch-2008-2}, see also the papers of Totik\,\cite{Totik-2010} and Widom\,\cite{Widom}.
\item If $S$ is the union of two real intervals, sharp estimates for $\kappa(S)$ in terms of elementary functions of the endpoints of the intervals are given in \cite{Sch-2010}.
\end{enumerate}
\end{remark}


From the proof of Theorem\,\ref{Thm-LB}, we obtain a refinement for the Bernstein--Walsh lemma \cite[Thm.\,5.5.7]{Ransford} for the case of several real intervals and real arguments.


\begin{corollary}\label{Cor}
Let $K$ be the union of a finite number of real intervals. Then for any polynomial $Q_n$ of degree $n$,
\[
\frac{|Q_n(x)|}{\|Q_n\|_K}\leq\frac{1}{2}\bigl(\exp(n\cdot{g}(x;K\cc))+\exp(-n\cdot{g}(x;K\cc))\bigr)\qquad(x\in\R\setminus{K}),
\]
where $g(x;K\cc)$ denotes the Green's function of $K\cc$. Equality is attained if and only if $Q_n$ is such that $Q_n^{-1}([-1,1])=K$.
\end{corollary}
\begin{proof}
Let $\xi\in\R\setminus{K}$ be fixed, and let $R_n(\,\cdot\,;\xi)\in\P_n$ be uniquely defined by
\[
\|R_n(\,\cdot\,;\xi)\|_K=\min\bigl\{\|P_n\|_K:P_n\in\P_n,P_n(\xi)=1\bigr\},
\]
i.e., $R_n(\,\cdot\,;\xi)$ is the minimal residual polynomial for the degree $n$ on $K$ with respect to $\xi$. We claim that for any polynomial $P_n\in\P_n$,
\begin{equation}\label{Ineq-Cor}
\frac{|P_n(\xi)|}{\|P_n\|_K}\leq\frac{1}{\|R_n(\,\cdot\,;\xi)\|_K}.
\end{equation}
Indeed, suppose on the contrary that there exists a polynomial $U_n\in\P_n$ such that $U_n(\xi)/\|U_n\|_K>1/\|R_n(\,\cdot\,;\xi)\|_K$. Define $\tilde{U}_n(x):=U_n(x)/U_n(\xi)$. Then $\tilde{U}_n(\xi)=1$ and, by definition of $R_n(\,\cdot\,;\xi)$, $\|\tilde{U}_n\|_K=\|U_n\|_K/U_n(\xi)<\|R_n(\,\cdot\,;\xi)\|_K$, which is a contradiction, and inequality \eqref{Ineq-Cor} is true.

Now by shifting $K$ to $S:=K-\xi$ it is clear that $L_n(S)=\|R_n(\,\cdot\,;\xi)\|_K$ and, for the corresponding Green's functions, $g(z+\xi;K\cc)=g(z;S\cc)$ for $z\in{S}\cc$. Thus, by the proof of Theorem\,\ref{Thm-LB},
\begin{align*}
\frac{1}{\|R_n(\,\cdot\,;\xi)\|_K}=\frac{1}{L_n(S)}
&\leq\frac{1}{2}\bigl(\exp(n\cdot{g}(0;S\cc))+\exp(-n\cdot{g}(0;S\cc))\bigr)\\
&=\frac{1}{2}\bigl(\exp(n\cdot{g}(\xi;K\cc))+\exp(-n\cdot{g}(\xi;K\cc))\bigr),
\end{align*}
which together with \eqref{Ineq-Cor} gives the assertion.
\end{proof}

{\bf Acknowledgement.} The author would like to thank Vilmos Totik for pointing out that Theorem\,\ref{Thm-LB} implies a refinement of the Bernstein--Walsh lemma given in Corollary\,\ref{Cor}.


\bibliographystyle{amsplain}

\bibliography{LowerBoundRn}

\end{document}